\documentclass[12pt]{amsart}
\usepackage{amsfonts}
\usepackage{amsthm}
\usepackage{graphics}
\usepackage{indentfirst}
\usepackage{cite}
\usepackage{latexsym}
\usepackage{amsmath}
\usepackage{amssymb}
\usepackage[dvips]{epsfig}
\usepackage{amscd}

\newtheorem{theorem}{Theorem}[section]
\newtheorem{remark}{Remark}[section]

\newtheorem{lemma}[theorem]{Lemma}
\newtheorem{pro}{Proposition}[section]
\newtheorem{cor}[theorem]{Corollary}

\renewcommand{\div}{ {\rm div }  }

\newcommand{\bt}{\begin{theorem}}
\newcommand{\bl}{\begin{lemma}}
\newcommand{\el}{\end{lemma}}
\newcommand{\et}{\end{theorem}}
\newcommand{\br}{\begin{remark}}
\newcommand{\er}{\end{remark}}

\newcommand{\curl}{{\rm curl} }

\newcommand{\la}{\label}

\newcommand{\bn}{\begin{eqnarray}}
\newcommand{\en}{\end{eqnarray}}
\newcommand{\bnn}{\begin{eqnarray*}}
\newcommand{\enn}{\end{eqnarray*}}

\newcommand{\ba}{\begin{aligned}}
\newcommand{\ea}{\end{aligned}}
\newcommand{\be}{\begin{equation}}
\newcommand{\ee}{\end{equation}}

\def\norm[#1]#2{\|#2\|_{#1}}

\def\la{\label}

\begin{document}

\title{Global Strong Solution With Vacuum to the 2D Nonhomogeneous Incompressible MHD System}

\author{Xiangdi Huang}
\address{Academy of Mathematics and Systems Science, CAS, Beijing 100190, P. R. China \ \ $\&$\ \ Department of Pure and Applied Mathematics, Graduate School of Information Sciences and Technology, Osaka University, Osaka, Japan}
\email{xdhuang@amss.ac.cn}

\author{Yun Wang}
\address{Department of Mathematics and Statistics, McMaster University, 1280 Main Street West, Hamilton, Ontario L8S 4K1, Canada}
\email{yunwang@math.mcmaster.ca}
\date{}
\maketitle

\begin{abstract}In this paper, we first prove the unique global strong solution with vacuum to the two dimensional nonhomogeneous incompressible MHD system, as long as the initial data satisfies some compatibility condition. As a corollary, the global existence of strong solution with vacuum to the 2D nonhomogeneous incompressible Navier-Stokes equations is also established. Our main result improves all the previous results where the initial density need to be strictly positive. The key idea is to use some critical Sobolev inequality of logarithmic type, which is originally due to Brezis-Wainger \cite{Brezis-Wainger}.

Keywords: nonhomogeneous incompressible fluid, strong solution, vacuum.

AMS: 35Q35, 35B65, 76N10
\end{abstract}

\section{Introduction}
Magnetic fields influence many fluids. Magnetohydrodynamics(MHD) is concerned with the interaction between fluid flow and magnetic field. The governing equations of nonhomogeneous MHD can be stated as follows\cite{Davidson},
\be\la{MHD} \left\{  \ba & \rho_t + \div (\rho u) = 0,\ \ \ \mbox{in}\ \Omega\times [0, T),\\
 & (\rho u)_t  + \div (\rho u\otimes u) - \div(2\mu(\rho) d) - (B\cdot \nabla )B + \nabla P = 0,\ \ \mbox{in}\ \Omega\times [0, T),\\
&B_t -\lambda \Delta B - \curl (u \times B) =0,\ \ \ \ \ \mbox{in}\ \Omega\times [0, T),\\
& {\rm div} u =0, \ \ \ \ \div B=0, \ \ \ \ \mbox{in}\ \Omega\times [0, T).\ea
\right.
\ee
Here $\rho$ and $u$ are the density and velocity field of fluid respectively. $P$ is the pressure. $B$ is the magnetic field.  $\mu(\rho)\geq 0$ denotes the viscosity of fluid, which we assume in this paper is a positive constant. $\lambda>0$ is also a constant, which describes the relative strengths of advection and diffusion of $B$. For simplicity of writing, let $\mu = \lambda=1$, $d= \frac12 \left(\nabla u + (\nabla u)^t\right)$ is the deformation tensor.

In this paper, we focus on the system (\ref{MHD}) 
with the initial-boundary conditions
\be\la{boundary-condition}  u = 0,\ \ \ \ B\cdot\vec{n} = 0, \ \ \ {\rm curl}B = 0\ \ \ \mbox{on}\ \partial \Omega\times [0, T),
\ee
\be\la{initial-condition}  (\rho, u, B)|_{t=0} = (\rho_0, u_0, B_0)\ \ \mbox{in}\ \Omega.
\ee
Here $\Omega$ is a bounded smooth domain in $\mathbb{R}^2$.

If there is no magnetic field, i.e., $B=0$, MHD system turns to be nonhomogeneous Navier-Stokes system. In fact, due to the similarity of the second equation and the third equation in \eqref{MHD}, the study for MHD system has been along
with that for Navier-Stokes one. Let's recall some known results for 3D nonhomogeneous Navier-Stokes equations. When the initial density $\rho_0$ is bounded away from 0, the global existence of weak solutions was established by Kazhikov\cite{Kazhikov}, see also \cite{AK}.  Moreover, Antontsev-Kazhikov-Monakhov\cite{AK2} gave the first result on local existence and uniqueness of strong solutions. For the two-dimensional case, they even proved that the strong solution is global.  But the global existence of strong or smooth solutions in 3D is still an open problem. For more results in this direction, see \cite{LS, Salvi, gui-zhang} and references therein.

If the initial density $\rho_0$ allows vacuum, the problem becomes more complicated. Simon\cite{Simon} proved the global existence of weak solutions, see also \cite{Lions}. Choe-Kim\cite{Choe-Kim} constructed a local strong solution under some compatibility conditions on the initial data. More precisely,
they proved that if $(\rho_0$, $u_0)$ satisfy
\be\la{initial-data-NS}
0\leq \rho_0 \in L^{\frac32}(\Omega)\cap H^2(\Omega),\ \ \ u_0\in D_{0}^1(\Omega)\cap D^2(\Omega),
\ee
and the compatibility conditions
\be \la{compatibility-ns}
 {\rm div}u_0=0,\ \ \ \ -\mu \Delta u_0 + \nabla P_0 = \rho_0^{\frac12} g,\ \ \ \mbox{in}\ \Omega,
\ee
with some $(P_0, g)$ belonging to $D^1(\Omega) \times L^2(\Omega)$, then there exists a positive time $T$ and a unique strong solution $(\rho, u)$ $\in C([0, T); H^2(\Omega)) \times C ([0, T);  D_0^1(\Omega) \cap D^2(\Omega))$ to the nonhomogeneous Navier-Stokes equations, where $D_0^1(\Omega)$ and $D^2(\Omega)$ denote the usual homogeneous Sobolev spaces. Recall that $D_0^{1}(\mathbb{R}^3)=
\{u\in L^6(\mathbb{R}^3):\nabla u\in L^2(\mathbb{R}^3)\}$ and
$D_0^1(\Omega)=H_0^1(\Omega)$ if $\Omega\subset\subset\mathbb{R}^3$.

After the local existence of strong solution, one question came out naturally, which is whether the solution blows up in finite time. Suppose the finite blow-up time $T^*$ exists, \cite{Kim} proved the Serrin type
criterion, which says that
\be \la{serrin-criterion}
\int_0^{T^*} \|u(t)\|_{L^r_w}^s dt = \infty,\ \ \ \mbox{for any}\ (r,s) \ \mbox{with}\ \frac{2}{s} + \frac{n}{r}= 1, \ \ n<r\leq \infty,
\ee
where $n$ is the dimension of the domain and $L_w^r$ is the weak $L^r$ space. (The proof was given in \cite{Kim} only for 3D case, but almost the same proof works for 2D case.)  In particular, for the 2D case, it follows from the energy
inequality the solution satisfies that $\sup_{0< T< T^*}(\|\sqrt{\rho} u\|_{L^\infty(0, T; L^2)} + \|\nabla u\|_{L^2(0, T; L^2)} )$ is bounded, which implies that $u\in L^4(0, T^*; L^4)$ if $\rho$ is bounded away from 0. Hence the criterion (\ref{serrin-criterion}) in fact implies global existence of strong solution provided that $\rho_0$ is bounded away from 0. However, if the density is allowed to vanish, whether the strong solution exists globally remains unknown. This is the main problem we shall address in this paper.

Let's go back to the MHD system (\ref{MHD}). As said before, the research for MHD goes along with that for Navier-Stokes equations. The results are similar. When $\rho$ is a constant, which means the fluid is homogeneous, the MHD system has been extensively studied. Duraut-Lions\cite{Duraut-Lions} constructed a class of weak solutions with finite energy and a class of local strong solutions. In particular, the 2D local strong solution has been proved to be global and unique. While for the three-dimensional case, different Serrin type criteria similar to \eqref{serrin-criterion} were given in \cite{He-Xin, He-Wang, Cao-Wu, Zhou-Gala}. As for the 3D Navier-Stokes equations, whether the local strong solution is global is still open.

When the fluid is nonhomogeneous, Gerbeau-Le Bris\cite{GLe}, Desjardins-Le Bris\cite{DLe} studied the global existence of weak solutions of finite energy in the whole space or in the torus. Global existence of strong solutions with small initial data in some Besov spaces was considered by Abidi-Paicu\cite{Abidi-Paicu}. Moreover, \cite{Abidi-Paicu} allowed variable viscosity and conductivity coefficients but required an essential assumption that there is no vacuum (more precisely, the initial data are closed to a constant state). Chen-Tan-Wang\cite{Chen-Tan-Wang} extended the local existence in presence of vacuum. In conclusion, if the initial data satisfies that
\be \la{initial-conditions}
0\leq \rho_0 \in H^2, \ \ \ \   \ (u_0, B_0)\in H^2,
\ee
and the compatibility conditions
\be \la{compatibility-conditions}
\ba
&u_0= 0, \ \ \ B_0\cdot \vec{n}=0,\ \ \ {\rm curl}B_0 = 0,\ \ \ \mbox{on}\ \partial \Omega,\\
&{\rm div}u_0 = {\rm div}B_0 = 0,\ \ \ \ -\Delta u_0 + \nabla P_0 - (B_0\cdot \nabla )B_0= \rho_0^{\frac12}g,\ \ \ \ \mbox{in}\ \Omega,
\ea\ee
with some $(P_0, g)\in H^1\times L^2$, then there exist a positive time $T$ and a unique strong solution $(\rho, u, B)$ to the problem \eqref{MHD}-\eqref{initial-condition}, such that
\be\la{strong-solution} \ba &\rho \in C([0, T]; H^2), \ \ \ (u, B) \in C([0, T]; H^2),\\
&  \ p\in C([0, T]; H^1)\cap L^2(0, T; H^2),\ \ \  (u_t, B_t) \in L^2(0, T; H^1), \\
& \mbox{and}\ \  (\rho_t, \sqrt{\rho}u_t, B_t) \in L^\infty(0, T; L^2).
\ea
\ee
For all the techniques, refer to \cite{Cho-Kim}.

It comes to the question whether the local strong solution blows up. After the proof of \cite{Kim} for nonhomogeneous Navier-Stokes equations, one can get the same criterion \eqref{serrin-criterion} for nonhomogeneous MHD, see also \cite{Zhou-Fan}. In particular, for the 2D case, it says that $\|u\|_{L^2_t L^\infty_x}$ becomes unbounded once the local strong solution blows up.  On the other hand, the energy inequality tells us $\|\nabla u\|_{L^2_t L^2_x}$ is uniformly bounded, which only imply that $\|u\|_{L^2_t ( BMO_x)}$ is uniformly bounded. Therefore, in view of the blowup criterion (\ref{serrin-criterion}), it's not enough to extend the local strong solution to global one. To improve the regularity of the velocity, we choose to apply a critical Sobolev inequality of logarithmic type, which is originally due to Brezis-Gallouet\cite{Brezis-Gallouet} and Brezis-Wainger\cite{Brezis-Wainger}. In this paper, we use some extension, which was proved by Ozawa\cite{Ozawa}. For a new proof, see \cite{Kozono-Ogawa-Taniuchi}. The inequality is stated as follow,
\bl
Assume $f\in H^1(\mathbb{R}^2) \cap W^{1, q}(\mathbb{R}^2)$, with some $q> 2$. Then it holds that
\be\la{critical-inequality}
\|f\|_{L^\infty(\mathbb{R}^2)} \leq C\left( 1 + \|\nabla f\|_{L^2(\mathbb{R}^2)} \left( \ln^+ \|f\|_{W^{1, q}(\mathbb{R}^2)} \right)^{\frac12} \right),
\ee
with some constant $C$ depending only on $q$.
\el

The same proof with some proper extension theorem(see \cite{Adam}),  in fact gives the following modified inequality, which involves the integral with respect to time. For 
completeness, we will give the proof in Section 2.
\bl \la{critical-inequality-lemma} Assume $\Omega$ is a bounded smooth domain in $\mathbb{R}^2$ and $f\in L^2(s, t; H^1(\Omega))\cap L^2(s, t; W^{1, q}(\Omega))$, with some $q>2$ and $0\leq s <t \leq \infty$. Then it holds that
\be\la{critical-inequality-time}
\|f\|_{L^2(s, t; L^\infty(\Omega))} \leq C\left(1 + \|f\|_{L^2(s, t;H^1(\Omega)) } \left(\ln^+ \|f\|_{L^2(s, t; W^{1, q}(\Omega))}   \right)^{\frac12}\right),
\ee
with some constant $C$ depending only on $q$ and $\Omega$, and independent of $s, t$.
\el

The application of\eqref{critical-inequality-time} is the key idea of this paper. Due to this, we can close the estimates for $\|(u, B)\|_{L^\infty_t H^1_x}$. The higher order estimates are in the same spirit of \cite{Kim}. For more details, see Section 3.  Finally, we get the result about global existence of strong solution.
\bt\la{main-result}  Assume that the initial data $(\rho_0, u_0, B_0)$ satisfies \eqref{initial-conditions} and the compatibility conditions \eqref{compatibility-conditions}. Then there exists a global strong solution
$(\rho, u, B)$ of the MHD system \eqref{MHD}-\eqref{initial-condition}, with
\be\la{global-solution}
\ba
& \rho \in C([0, \infty); H^2), \ \ \ (u, B) \in C([0, \infty); H^2),\\ & P\in C([0, \infty) ; H^1)\cap L^2_{loc}(0, \infty; H^2), \ \ \
 (u_t, B_t) \in L^2_{loc}(0, \infty; H^1), \\
& \mbox{and}\ \  (\rho_t, \sqrt{\rho}u_t, B_t) \in L^\infty_{loc}(0, \infty ; L^2).
\ea
\ee
\et

Some remarks are given about this theorem.
\br
The local existence of unique strong solution with vacuum to the system (\ref{MHD})  in a two-dimensional bounded domain can be established in the same manner as \cite{Choe-Kim} and \cite{Chen-Tan-Wang}. Through this paper, we will concentrate on establishing global estimates for the density, velocity and magnetic field.
\er

\br
 If we consider the most special case, where $\rho$ is a constant(the fluid is homogeneous) and $B=0$(no magnetic field), then the system \eqref{MHD} becomes the classical Navier-Stokes system. The global
existence of strong solution has been proved by Leray\cite{Leray}. More generally, if we consider the case that only $\rho$ is a constant, the system \eqref{MHD} becomes the classical homogeneous MHD system. As said before, the corresponding result has been derived by Duraut-Lions\cite{Duraut-Lions}.
\er

If $B=0$, Theorem \ref{main-result} in fact gives a positive answer to the global existence of strong solutions with vacuum of the 2D nonhomogeneous Navier-Stokes system. It covers the corresponding result in \cite{AK2}, where the density is strictly positive.

\begin{cor}
  Assume that the initial data $(\rho_0, u_0)$ satisfies \eqref{initial-conditions} and the compatibility conditions \eqref{compatibility-ns}. Then there exists a global strong solution
$(\rho, u)$ of the Navier-Stokes equations, with
\be\la{global-solution-ns}
\ba
& \rho \in C([0, \infty); H^2), \ \ \ u \in C([0, \infty); H^2),\\ & P\in C([0, \infty) ; H^1)\cap L^2_{loc}(0, \infty; H^2),\ \ \
 u_t \in L^2_{loc}(0, \infty; H^1), \\  &\mbox{and}\ \  (\rho_t, \sqrt{\rho}u_t) \in L^\infty_{loc}(0, \infty ; L^2).
\ea
\ee
\end{cor}


We conclude this section with some notations and lemmas.
$L^r(\Omega), W^{k, r}(\Omega)$, $(1\leq r\leq \infty)$, are the standard Sobolev spaces, and we use $L^r=L^r(\Omega)$, $W^{k, r}= W^{k,r}(\Omega)$.
Especially, when $r=2$, denote $H^k = W^{k, 2}$.
For simplicity, let
$$\int f dx \triangleq \int_{\Omega} f dx. $$

Some more lemmas will be used during the proof of Theorem \ref{main-result}. One is following from the regularity theory for Stokes equations. For its proof, refer to \cite{Galdi}.
\bl\la{Galdi} Assume that $(u, P)\in H_0^1 \times H^1$ is a weak solution of the stationary Stokes equations,
\be\la{Stokes}\left\{
\ba
& -\Delta u + \nabla P = F,\ \ \ \mbox{in}\ \Omega,\\
&{\rm div}u = 0,\ \ \ \ \ \mbox{in}\ \Omega,\\
&u=0, \ \ \ \ \ \ \mbox{on}\ \partial\Omega,
\ea\right.
\ee
and $F\in L^q$, $1<q<\infty$. Then it holds that
\be
\|u\|_{W^{2, q}} \leq C\|F\|_{L^q} + C\|u\|_{H^1},
\ee
with some constant $C$ depending on $\Omega$ and $q$. Moreover, if $F\in H^1$, then
\be
\|u\|_{H^3} \leq C \|F\|_{H^1} + C\|u\|_{H^1},
\ee
with some constant $C$ depending only on $\Omega$.
\el

The other lemma is responsible for the estimates for $B$ and follows from the classical regularity theory for elliptic equations. For its proof, refer to \cite{Nirenberg}.
\bl \la{Nirenberg} Assume that $B\in H^1$ is a weak solution of the Poisson equations
\be\la{Poisson} \left\{ \ba
& \Delta B= G,\ \ \ \mbox{in}\ \Omega,\\
& B\cdot \vec{n}= 0, \ \ \ {\rm curl}B= 0, \ \ \ \mbox{on}\ \partial \Omega,
\ea
\right.
\ee
and $G\in L^q$, $1<q<\infty$. Then it holds that
\be
\|B\|_{W^{2, q}} \leq C \|G\|_{L^q} + C \|B\|_{H^1},
\ee
with some constant $C$ depending on $\Omega$ and $q$. Moreover, if $G\in H^1$, then
\be
\|B\|_{H^3} \leq C \|G\|_{H^1} + C\|B\|_{H^1},
\ee
with some constant $C$ depending only on $\Omega$.
\el

\section{Proof of Lemma 1.2}
This section is dedicated to the proof of Lemma 1.2. First we will prove the inequality \eqref{critical-inequality-time} for the whole space case, which is 
\be\la{critical-whole-space}
\|f\|_{L^2(s, t; L^\infty(\mathbb{R}^2))} \leq C\left(1 + \|f\|_{L^2(s, t;H^1(\mathbb{R}^2)) } \left(\ln^+ \|f\|_{L^2(s, t; W^{1, q}(\mathbb{R}^2))}   \right)^{\frac12}\right).
\ee
The proof follows exactly that in \cite{Kozono-Ogawa-Taniuchi} and lies mainly on the Littlewood-Paley decomposition. So we introduce here some new notations associated with the decomposition. 
Define $\mathcal{C}$ to be the ring
$$\mathcal{C}= \left\{ \xi \in \mathbb{R}^2:\ \ \frac34 \leq |\xi| \leq \frac83    \right\},$$
and define $\mathcal{D}$ to be the ball
$$\mathcal{D}= \left\{\xi \in \mathbb{R}^2: \ \ \  |\xi|\leq \frac43   \right\}.$$
Let $\chi$ and $\varphi$ be two smooth nonnegative radial functions supported respectively in $\mathcal{D}$ and $\mathcal{C}$, such that
$$\chi(\xi) + \sum_{q\in \mathbb{N}} \varphi(2^{-q} \xi)= 1\ \ \mbox{for}\ \xi \in \mathbb{R}^2,\ \ \mbox{and}\ \ \sum_{q\in \mathbb{Z}}\varphi(2^{-q}\xi )= 1\ \ \mbox{for}\ \xi \in \mathbb{R}^2 
\setminus \{0\}.$$
Denote the Fourier transform on $\mathbb{R}^2$ by $\mathcal{F}$ and denote
$$h= \mathcal{F}^{-1} \varphi,\ \ \ \ \ \ \tilde{h}= \mathcal{F}^{-1} \chi.$$
The frequency localization operator is defined by 
$$\Delta_q f = \mathcal{F}^{-1} \left[ \varphi(2^{-q} \xi )\mathcal{F} (f) \right] = 2^{2q} \int_{\mathbb{R}^2} h(2^q y) f (x-y ) dy,$$
and 
$$S_q f = \mathcal{F}^{-1} \left[ \chi(2^{-q}\xi ) \mathcal{F}(f)     \right]= 2^{2q} \int_{\mathbb{R}^2} \tilde{h}(2^q y) f(x-y) dy. $$

Now it's ready to prove \eqref{critical-whole-space}.
\begin{proof}
Decompose $f$ into three parts such as
\be \ba
f(x, \tau) & = S_{-N-1} f(x, \tau) + \sum_{|j|\leq N} \Delta_{j} f(x, \tau) + \sum_{j>N }\Delta_j f(x, \tau)\\
& = f_1(x, \tau) + f_2(x, \tau) + f_3(x, \tau).
\ea
\ee
By Bernstein's inequality(see \cite{Chemin}), 
\be
\|f_1\|_{L^2(s, t; L^\infty)}\leq   C^{-2N/q}\|f\|_{L^2(s, t; L^q)},\ \ \ q\in [1, \infty).
\ee
Similarly, 
\be\ba
\|f_2\|_{L^2(s, t; L^\infty) } & \leq  \sum_{|j|\leq N} \|\Delta_j f\|_{L^2(s, t; L^\infty)} \\
&  \leq C N^{\frac12} \left( \| \nabla (\Delta_j f )   \|_{L^2(s, t; L^2)}^2 \right)^{\frac12}\\
& \leq CN^{\frac12} \|\nabla f\|_{L^2(s, t; L^2)},
\ea\ee
and 
\be \ba
\|f_3\|_{L^2(s, t; L^\infty)} & \leq \sum_{j > N } \|\Delta_j  f\|_{L^2(s, t; L^\infty)} \\
& \leq  C\sum_{j>N} 2^{2j (1/q - 1/2)} \|\nabla f\|_{L^2(s, t; L^q)}\\
& =  C 2^{(2/q- 1)N} \|\nabla f\|_{L^2(s, t; L^q)}.
\ea\ee

If we set $\kappa = \min (2/q, \ 2(1/2-1/q))$, then 
\be
\|f\|_{L^2(s, t; L^\infty)} \leq C \left\{  2^{-\kappa N} \|f\|_{L^2(s, t; W^{1,q})} + N^{\frac12} \|\nabla f\|_{L^2(s, t; L^2)}  \right\}.
\ee

Choose $N = \left[ \log_{2^\kappa} \frac{ \|f\|_{L^2(s, t; W^{1,q}) }}{\|\nabla f\|_{L^2(s, t; L^2)}}\right]+ 1$, hence
we derive that
\be
\|f\|_{L^2(s, t; L^\infty)} \leq C \|\nabla f\|_{L^2(s, t; L^2)} \left( 1 + \left( \ln^+ \frac{ \|f\|_{L^2(s, t; W^{1,q}) }}{\|\nabla f\|_{L^2(s, t; L^2)}}         \right)^{1/2} \right),
\ee
which implies \eqref{critical-whole-space}.
\end{proof}

Combining the extension theorem(see \cite{Adam}) and \eqref{critical-whole-space}, we prove Lemma 1.2.

\section{Proof of Theorem 1.3}
This section is dedicated to the proof of Theorem 1.3. Define the quantity $\Phi(T)$ as follow,
\be\la{Phi}\ba
\Phi(T)& =  \sup_{0\leq t\leq T}\left(\| \rho(t)\|_{H^2}^2 + \|u(t)\|_{H^2}^2 + \|B(t)\|_{H^2}^2 \right) + \|\sqrt{\rho} u_t\|_{L^\infty(0, T; L^2)}^2\\
& \ \ \ + \int_0^T \left(\|u(t)\|_{H^3}^2 + \|B(t)\|_{H^3}^2 \right)dt+ \int_0^T\left( \|u_t\|_{H^1}^2 + \|B_t\|_{H^1}^2\right)dt.
\ea\ee

Suppose the local strong solution blows up at $T^*<\infty$, we will prove that in fact there exists a generic constant $\bar{M}<\infty$ depending only the initial data and $T^*$ such that
\be \la{final} \sup_{0\leq T < T^*} \Phi(T)\le\bar{M}. \ee
 Having (\ref{final}) at hand, it is easy to show without many difficulties that we can extend the strong solution beyond $T^*$, which gives a contradiction. Hence the local strong solution does not blow up in finite time. Also, the uniqueness of strong solutions is a standard procedure. 
 
 Through out this section, $C$  denote a generic constant only depending on the initial data and $T^*$. The proof is divided into five steps, due to different level estimates.

Before proceeding, we write another equivalent form of \eqref{MHD} for convenience, which is
\be\la{MHD-new}\left\{\ba & \rho_t + u\cdot \nabla \rho = 0,\\
&\rho u_t -  \Delta u + (\rho u\cdot \nabla )u - (B\cdot \nabla ) B + \nabla P =0, \\
&B_t - \Delta B +  (u\cdot \nabla ) B - (B\cdot \nabla )u = 0,\\
&\div u =0, \ \ \ \div B =0.
\ea
\right. \ee

Now we start the proof of Theorem \ref{main-result}.

\vspace{2mm}{\bf Step I\ \ $L^\infty$ bound for $\rho$.}\ The equation $\eqref{MHD-new}_1$ for density is  a transport equation, then for every $0\leq t< T^*$,
\be \la{density}
\|\rho(t)\|_{L^\infty} = \|\rho_0\|_{L^\infty}.
\ee

\vspace{2mm}{\bf Step II\ \ Basic energy estimate}
\begin{pro}[Energy inequality] There exists a constant $M$ depending only on $\|\sqrt{\rho_0} u_0\|_{L^2}$ and $\|B_0\|_{L^2}$, such that for every $0< T< T^*$,
\be\la{energy-estimate}\|\sqrt{\rho} u \|_{L^\infty(0, T; L^2 )}^2 + \|B\|_{L^\infty(0, T; L^2)}^2  + \int_0^T  \|\nabla u\|_{L^2}^2 dt +   \int_0^T \|\nabla B\|_{L^2}^2 dt
\leq M.\ee
\end{pro}
\begin{proof} The proof is standard.
Multiplying $\eqref{MHD-new}_2$ and $\eqref{MHD-new}_3$ by $u$ and $B$ respectively, then adding the two resulting equations together, integrating over $\Omega$,  one can get that
\be
\frac{1}{2}\frac{d}{dt} \int \rho |u|^2dx +\frac12\frac{d}{dt} \int |B|^2 dx + \int |\nabla u|^2 dx  + \int |\nabla B|^2 dx =0,
\ee
where integration by parts was applied. It implies that the inequality \eqref{energy-estimate} holds and consequently completes the proof.
\end{proof}

{\bf Step III\ \ Estimates for $\|(\sqrt{\rho}u_t,\ B_t)\|_{L^2(0, T;  L^2)}$ and $\|(\nabla u, \nabla B)\|_{L^\infty(0, T; L^2)}.$}

 This is a crucial step during the proof. Higher order estimates of the density, velocity and magnetic field can be done in a standard way provided that
$\|(u,\ B)\|_{H^1}$ is uniformly bounded with respect to time. To prove that, we will make use of some extension of critical Sobolev inequality of logarithmic type, as indicated by Lemma \ref{critical-inequality-lemma}.

\begin{pro} \la{First-Level}Under the assumptions in Theorem \ref{main-result}, it holds that
\be\la{first-level}
\sup_{0<T<T^*}\left \{\|(u(T),\  B(T))\|_{H^1}^2  + \int_0^T \|(\sqrt{\rho}u_t,\ B_t)\|_{L^2}^2 dt \right\}< \infty.
\ee
\end{pro}
\begin{proof}
Multiplying the equation $\eqref{MHD-new}_2$ by $u_t$ and integrating over $\Omega$ lead to
\be\la{first-level-1}
\frac12\frac{d}{dt} \int |\nabla u|^2 dx + \int \rho |u_t|^2 dx =
- \int (\rho u\cdot \nabla u)\cdot u_t dx + \int (B\cdot \nabla) B\cdot u_t dx.
\ee
By H\"older's inequality and Young inequality,
\be\la{first-level-2}\ba
\left| \int (\rho u \cdot \nabla) u \cdot u_t dx \right|&  \leq C\|\sqrt{\rho} u_t\|_{L^2 } \cdot \|u\|_{L^\infty} \cdot \|\nabla u\|_{L^2}\\
& \leq \frac{1}{2} \|\sqrt{\rho}u_t\|_{L^2}^2 + C\|u\|_{L^\infty}^2 \|\nabla u\|_{L^2}^2.
\ea
\ee
Applying integration by parts with the conditions that ${\rm div} B=0$ in $\Omega$ and $B\cdot \vec{n}= 0$ on $\partial \Omega$, then
\be \la{first-level-3}\ba &\int (B\cdot \nabla) B\cdot u_t dx \\
 =& \frac{d}{dt}\int (B\cdot \nabla ) B\cdot u dx - \int (B_t \cdot \nabla )B \cdot u dx - \int (B\cdot \nabla ) B_t \cdot u dx\\
=& -\frac{d}{dt} \int (B\cdot \nabla ) u \cdot B dx + \int (B_t \cdot \nabla ) u \cdot B dx + \int (B\cdot \nabla ) u \cdot B_t dx\\
\leq & - \frac{d}{dt} \int (B\cdot \nabla ) u \cdot B dx + C\|B\|_{L^\infty}^2\|\nabla u\|_{L^2}^2 + \frac12 \|B_t\|_{L^2}^2.\ea \ee
Hence, combining \eqref{first-level-1}-\eqref{first-level-3}, we get that
\be\la{first-level-4} \ba
&\frac12 \|\sqrt{\rho} u _t\|_{L^2}^2 + \frac12 \frac{d}{dt}\int |\nabla u|^2 dx + \frac{d}{dt}\int (B\cdot \nabla ) u \cdot B dx \\
\leq & C\left( \|u\|_{L^\infty}^2 + \|B\|_{L^\infty}^2\right) \|\nabla u\|_{L^2}^2 + \frac12 \|B_t\|_{L^2}^2.
\ea\ee

Similarly, multiplying the equation $\eqref{MHD-new}_3$ by $B_t$ and integrating over $\Omega$ lead to
\be\la{first-level-5}
\ba & \frac12 \frac{d}{dt} \int |\nabla B|^2 dx + \int |B_t|^2 dx  \\
 = & - \int (u \cdot \nabla B) \cdot B_t dx + \int (B\cdot \nabla ) u \cdot B_t dx \\
 \leq &  \frac12 \|B_t\|_{L^2}^2 + C\|u\|_{L^\infty}^2 \|\nabla B\|_{L^2}^2 + C\|\nabla u\|_{L^2}^2 \|B\|_{L^\infty}^2,
\ea
\ee
which implies that
\be \la{first-level-6}
 \frac{d}{dt}\int |\nabla B|^2 dx +  \|B_t\|_{L^2}^2 \leq C\|u\|_{L^\infty}^2 \|\nabla B\|_{L^2}^2 + C\|B\|_{L^\infty}^2 \|\nabla u\|_{L^2}^2.
\ee

The term $\int (B\cdot \nabla ) u \cdot B dx $ on the left hand of \eqref{first-level-4} can not be determined positive or negative, so  we choose some appropriate positive terms to control it.  Note that it follows from Gagliardo-Nirenberg inequality that
\be \la{first-level-7} \ba
 \left|\int (B\cdot \nabla ) u \cdot B dx \right|
 \leq & \|B\|_{L^4}^2 \|\nabla u\|_{L^2}\\
 \leq &  C\|B\|_{L^2} \|B\|_{H^1} \|\nabla u\|_{L^2} \\
\leq & \frac{1}{4}\|\nabla u \|_{L^2}^2 + C_1 \|B\|_{L^2}^2 (\|B\|_{L^2}^2 + \|\nabla B\|_{L^2}^2).
\ea\ee
 Next, we multiply \eqref{first-level-6} by $2C_1 M +2$, where $C_1$ and $M$ are constants appearing in \eqref{first-level-7} and \eqref{energy-estimate}, add it to \eqref{first-level-4} and integrate
with respect to time, then for every $0\leq s <T<T^* $,
\be\la{first-level-8} \ba
&\int |\nabla u(T) |^2 dx + \int |\nabla B(T)|^2 dx +  \int_s^T \|\sqrt{\rho}u_t\|_{L^2}^2 d\tau + \int_s^T \|B_t\|_{L^2}^2 d\tau \\
\leq &  C \left[ \int |\nabla u(s) |^2 dx + \int |\nabla B(s)|^2 dx\right] \exp \left\{ C\int_s^T (\|u\|_{L^\infty}^2 +  \|B\|_{L^\infty}^2) d\tau \right \} + C.
\ea \ee

Denote \be \la{first-level-9} \Psi (t) = e+\sup_{0\leq \tau \leq t} \left( \|u(\tau)\|_{H^1}^2 + \|B(\tau)\|_{H^1}^2\right) + \int_0^t \left( \|\sqrt{\rho}u_t\|_{L^2}^2  + \|B_t\|_{L^2}^2 \right) d\tau , \ee
then \eqref{first-level-8} and \eqref{energy-estimate} give that for every $0\leq s < T < T^*$,
\be\la{first-level-10}
\Psi(T) \leq C \Psi(s) \exp \left\{ C\int_s^T (\|u\|_{L^\infty}^2 +  \|B\|_{L^\infty}^2) d\tau \right \}.
\ee

To get a proper estimate for $\|u\|_{L_t^2 L_x^\infty}$ and $\|B\|_{L^2_t L_x^\infty}$, we get help from Lemma \ref{critical-inequality-lemma}.
\be \la{first-level-11} \ba
& \|u\|_{L^2(s, T; L^\infty) }^2  + \|B\|_{L^2(s, T; L^\infty)}^2 \\ \leq & C  \left\{ 1 + (\|u\|_{L^2(s, T; H^1)}^2 + \|B\|_{L^2(s, T; H^1)}^2 ) \left(\ln^+ \|u\|_{L^2(s, T; W^{1,4})}  + \ln^+ \|B\|_{L^2(s, T; W^{1,4})}\right)\right\}.
\ea
\ee
Applying Lemma \ref{Galdi} to the equation $\eqref{MHD-new}_2$ yields
\be\la{first-level-12}
\|u\|_{W^{1,4}} \leq C \|u\|_{H^1} + C \|\rho u_t \|_{L^{\frac43}} + C \|(\rho u\cdot \nabla)u - (B\cdot \nabla )B\|_{L^{\frac43}} ,
\ee
which implies
\be\la{first-level-13} \ba
\|u\|_{L^2(s, T; W^{1,4})}  \leq & C \|u\|_{L^2(s, T; H^1)} + C\|\sqrt{\rho} u_t\|_{L^2(s, T; L^2)}  \\ & + C \|u\|_{L^2(s, T; H^1)} \|\nabla u\|_{L^\infty(s, T; L^2)}
+ C \|B\|_{L^2(s, T; H^1)} \|\nabla B\|_{L^\infty(s, T; L^2)}. \ea \ee
Similarly, applying Lemma \ref{Nirenberg} to the equation $\eqref{MHD-new}_3$ to obtain
\be \la{first-level-14} \ba
 \|B\|_{L^2(s, T; W^{1,4})}  \leq & C \|B\|_{L^2(s, T; H^1)} + C\|B_t\|_{L^2(s, T; L^2)}  \\ & + C \|u\|_{L^2(s, T; H^1)} \|\nabla B\|_{L^\infty(s, T; L^2)}
+ C \|B\|_{L^2(s, T; H^1)} \|\nabla u\|_{L^\infty(s, T; L^2)}.
\ea \ee
Note that the constant $C$ in \eqref{first-level-13} and \eqref{first-level-14} does not depend on $u$, $B$, $s$ or $T$. It only depends on the domain $\Omega$. Taking the energy inequality
\eqref{energy-estimate} into consideration, then for every $0\leq s<T< T^*$,
\be\la{first-level-15}\ba
& \|u\|_{L^2(s, T; L^\infty)}^2 + \|B\|_{L^2(s, T; L^\infty)}^2 \\ \leq  & C_2 \left\{1 +( \|u\|_{L^2(s, T; H^1)}^2 + \|B\|_{L^2(s, T; H^1)}^2 )\ln \left(  C(M, T^*) \Psi(T)   \right) \right\},
\ea
\ee
where $C_2$ is constant which only depends on $\Omega$, and  $C(M, T^*)$ is a constant depending on $M$ in \eqref{energy-estimate} and $T^*$.

Substituting \eqref{first-level-15} into \eqref{first-level-10}, it arrives at
\be\la{first-level-16}
\Psi(T)  \leq  C \Psi(s) \left[ C(M, T^*) \Psi(T) \right]^{C_2 \left( \|u\|_{L^2(s, T; H^1)}^2 + \|B\|_{L^2(s, T; H^1)}^2 \right)}.
\ee
Recall the energy estimate (\ref{energy-estimate}), one can choose $s$ close enough to $T^*$, such that
\be \lim_{T\rightarrow T^*} C_2\left( \|u\|_{L^2(s, T; H^1)}^2 + \|B\|_{L^2(s, T; H^1)}^2 \right)  \leq \frac12, \ee
then for every $s<T< T^*$, we have
\be
\Psi(T) \leq  C \Psi(s)^2 \cdot C(M, T^*)^2,
\ee
which completes the proof of Proposition \ref{First-Level}.
\end{proof}

\br Unfortunately, we can not get any explicit bound for $\|(u, B)\|_{H^1}$ in terms of the initial data, due to the technique used here.
\er

We have some more estimates as corollaries of Proposition \ref{First-Level}.
\begin{pro}\la{first-level-add-1} Assume that \be\la{assumption} \sup_{0< T< T^*} \left\{ \|(u(T), B(T))\|_{H^1}^2 + \int_0^T \|(\sqrt{\rho}u_t, B_t)\|_{L^2}^2 dt\right\} \leq C_3.\ee
Then there exists a constant $C_4$ depending on $C_3$, such that
\be \la{first-level-11-1}
\sup_{0< T<T^*} \left\{ \|u\|_{L^2(0, T; H^2) }+ \|B\|_{L^2 (0, T; H^2)} \right\} \leq C_4.
\ee
\end{pro}

\begin{proof} The equation $\eqref{MHD-new}_2$, together with Lemma \ref{Galdi}, gives us that
\be\la{first-level-12-1}\ba
\|u\|_{H^2} &  \leq C\|u\|_{H^1} + C\|\rho u_t\|_{L^2} + C\|(\rho u \cdot \nabla ) u\|_{L^2} + C\| (B\cdot \nabla ) B\|_{L^2}\\
& \leq C\|u\|_{H^1}+ C\|\sqrt{\rho}u_t\|_{L^2} + C\|u\|_{L^\infty} \|\nabla u\|_{L^2} + C \|B\|_{L^\infty} \|\nabla B\|_{L^2}.
\ea\ee
Similarly, by Lemma \ref{Nirenberg},
\be \la{first-level-13-1}\ba
\|B\|_{H^2} \leq C\|B\|_{H^1} + C\|B_t\|_{L^2} + C\|u\|_{L^\infty}\|\nabla B\|_{L^2} + C\|B\|_{L^\infty} \|\nabla u\|_{L^2}. \ea
\ee
Combining the two inequalities \eqref{first-level-12-1} and \eqref{first-level-13-1}, we have
\be\la{first-level-14-1}
\ba
& \|u\|_{H^2} + \|B\|_{H^2} \\ \leq &  C \|\sqrt{\rho} u_t\|_{L^2} + C\|B_t\|_{L^2} + C \left(\|u\|_{L^\infty} + \|B\|_{L^\infty} + 1 \right)\cdot \left(\| u\|_{H^1} + \| B\|_{H^1}\right)\\
\leq &  C\left(\|u\|_{H^2}  + \|B\|_{H^2}\right)^{1/2} \left(\|u\|_{L^2}+ \|B\|_{L^2}\right)^{1/2}  \cdot \left(\| u\|_{H^1} + \|B\|_{H^1}\right)\\
&\  + C \left(  \| u\|_{H^1} + \|B\|_{H^1} \right)+ C \|\sqrt{\rho} u_t\|_{L^2} + C\|B_t\|_{L^2}.
\ea\ee
where Gagliardo-Nirenberg inequality was used. Hence,
\be\la{first-level-15-1}
\|u\|_{H^2}+ \|B\|_{H^2} \leq C\|\sqrt{\rho} u_t\|_{L^2} + C\|B_t\|_{L^2} + C \left(1+ \|u\|_{H^1} + \|B\|_{H^1}\right)^3,
\ee
which completes the proof for \eqref{first-level-11-1}.
\end{proof}

\begin{pro}\la{first-level-add-2}Assume \eqref{assumption} holds, then there exists some constant $C_5$ depending on $C_3$ such that
\be \la{first-level-16-1}  \sup_{0<T < T^*} \left\{ \|u\|_{L^4(0, T; L^\infty)} + \|B\|_{L^4(0, T; L^\infty)} \right\} \leq C_5.
\ee
\end{pro}
\begin{proof}
By Gagliardo-Nirenberg inequality,
\be
\|u\|_{L^\infty} \leq C\|u\|_{L^2}^{1/2} \cdot \|u \|_{H^2}^{1/2},
\ee
and
\be
\|B\|_{L^\infty} \leq C \|B\|_{L^2}^{1/2} \cdot \|B\|_{H^2}^{1/2},
\ee
which together with \eqref{first-level-11-1} completes the proof for \eqref{first-level-16-1}.
\end{proof}

\vspace{2mm} {\bf Step IV\ \ Estimates for $\|(\sqrt{\rho}u_t, \ B_t)\|_{L^\infty(0, T; L^2)}$ and $\|(\nabla u_t, \nabla B_t)\|_{L^2(0, T; L^2)}$} From now on, the estimates are standard, due to the
proof in \cite{Kim}. We write them down here for completeness.
\begin{pro}Under the assumptions in Theorem \ref{main-result}, it holds that
 \la{Second-Level}
\be\la{second-level}
\sup_{0<T<T^*}\left \{\|(\sqrt{\rho} u_t(T), \ B_t(T))\|_{H^1} + \int_0^T \|(\nabla u_t, \ \nabla B_t)\|_{L^2}^2 dt \right\}< \infty.
\ee
\end{pro}
\begin{proof} Taking $t$-derivative of the equation $\eqref{MHD-new}_2$, then one gets that
\be\la{second-level-1}\ba
&\rho u_{tt} + (\rho u\cdot \nabla ) u_t - \Delta u_t + \nabla P_t \\ & = -\rho_t u_t - (\rho_t u \cdot \nabla ) u  - (\rho u_t \cdot \nabla ) u + (B_t \cdot \nabla ) B + (B\cdot \nabla ) B_t.
\ea \ee
Multiplying \eqref{second-level-1} by $u_t$ and integrating over $\Omega$,
\be\la{second-level-2}\ba
& \frac12 \frac{d}{dt} \int \rho |u_t|^2 dx + \int |\nabla u_t|^2 dx
= - \int \rho_t |u_t|^2 dx - \int (\rho_t u \cdot \nabla )u \cdot u_t dx \\ & \ \ \ \  - \int (\rho u_t \cdot \nabla )u \cdot u_t dx + \int(B_t \cdot \nabla )B\cdot u_t dx + \int (B\cdot \nabla) B_t \cdot u_t dx.
\ea
\ee
We estimate the terms on the right hand one by one. Taking $\eqref{MHD}_1$ into consideration, we get that
\be\la{second-level-3}\ba
-\int \rho_t |u_t|^2 dx  & = \int {\rm div}(\rho u) |u_t|^2 dx \\
& = -\int 2\rho u \cdot \nabla u_t \cdot u_t dx \\
& \leq \frac18 \|\nabla u_t \|_{L^2}^2 + C\|\sqrt{\rho}u_t\|_{L^2}^2 \|u\|_{L^\infty}^2,
\ea \ee
and also for the second term,
\be \la{second-level-4}\ba
& -\int (\rho_t u\cdot \nabla ) u \cdot u_t dx\\
 =& -\int \rho u \cdot \nabla [(u\cdot \nabla) u\cdot u_t]dx\\
\leq & \int |\rho u_t| |u| |\nabla u|^2 dx + \int |\rho u_t | |u|^2 |\nabla ^2 u| dx + \int \rho |u|^2 |\nabla u| |\nabla u_t| dx\\
 \ea\ee
Here by Gagliardo-Nirenberg inequality,
\be\la{second-level-5}\ba
& \int |\rho u_t| |u| |\nabla u|^2 dx \\
& \leq  \|\sqrt{\rho }u_t\|_{L^2}\|u\|_{L^\infty} \|\nabla u\|_{L^4}^2 \\
& \leq C\|\sqrt{\rho}u_t \|_{L^2}\|u\|_{L^\infty}  \|\nabla u\|_{L^2} \|\nabla u\|_{H^1} \\
& \leq  \|u\|_{L^\infty}^2 \|\sqrt{\rho}u_t\|_{L^2}^2  + C\|\nabla u\|_{L^2}^2 \| u\|_{H^2}^2.
\ea \ee
By Young inequality,
\be\la{second-level-6}\ba
&\int |\rho u_t | |u|^2 |\nabla ^2 u| dx \\
& \leq C\|\sqrt{\rho} u_t\|_{L^2} \|u\|_{L^\infty}^2 \|\nabla^2 u\|_{L^2} \\
& \leq \|u\|_{L^\infty}^4 \|\sqrt{\rho} u_t\|_{L^2}^2 + C\| u\|_{H^2}^2. \\
\ea\ee
And similarly,
\be\la{second-level-7}\ba
& \int \rho |u|^2 |\nabla u| |\nabla u_t| dx\\
  \leq &   C\|u\|_{L^\infty}^2 \|\nabla u\|_{L^2}\|\nabla u_t\|_{L^2}\\
 \leq & \frac18 \|\nabla u_t\|_{L^2}^2 + C\|u\|_{L^\infty}^4 \|\nabla u\|_{L^2}^2.
\ea \ee

For the third term of the right hand of \eqref{second-level-2}, by Poincar\'{e} inequality and Gagliardo-Nirenberg inequality,
\be\la{second-level-8} \ba  & -\int (\rho u_t \cdot \nabla ) u \cdot u_t dx \\
 \leq  &C\|\sqrt{\rho} u_t \|_{L^2} \|\nabla u\|_{L^4} \|u_t\|_{L^4} \\
 \leq & C \|u\|_{H^2}^2 \|\sqrt{\rho}u_t\|_{L^2}^2 + \frac18 \|\nabla u_t\|_{L^2}^2.
\ea \ee

Since ${\rm div}B_t =0 $ in $\Omega$ and $B_t \cdot \vec{n}= 0$ on $\partial \Omega$, then
\be \la{second-level-9} \ba & \int (B_t \cdot \nabla ) B \cdot u_t dx\\
 = & -\int (B_t \cdot \nabla) u_t \cdot B dx\\
 \leq & \frac18 \|\nabla u_t \|_{L^2 }^2 + C \|B\|_{L^\infty}^2 \|B_t\|_{L^2}^2.
\ea \ee

And similarly,
\be \la{second-level-10} \ba & \int (B\cdot \nabla ) B_t \cdot u_t dx \\
\leq &  \frac18 \|\nabla u_t\|_{L^2}^2+ C\|B\|_{L^\infty}^2 \|B_t\|_{L^2}^2.
\ea \ee

Now we turn to the equation for $B$. Taking $t$-derivative of $\eqref{MHD-new}_3$, multiplying by $B_t$ and integrating over $\Omega$, then
\be \la{second-level-12} \ba & \frac12 \frac{d}{dt} \int |B_t|^2dx + \int |\nabla B_t|^2 dx \\
= & -\int (u_t \cdot \nabla ) B\cdot B_t dx +  \int (B_t \cdot \nabla) u\cdot B_t dx + \int (B\cdot \nabla ) u_t \cdot B_t dx\\
\ea \ee
Here Poincar\'{e} inequality gives that
\be \la{second-level-13}\ba &-\int (u_t \cdot \nabla ) B\cdot B_t dx \\
\leq & \|u_t\|_{L^4}\|\nabla B\|_{L^4} \|B_t\|_{L^2}\\
\leq &  \frac18 \|\nabla u_t\|_{L^2}^2 + C \|\nabla B\|_{H^1}^2\|B_t\|_{L^2}^2 .
\ea\ee
Gagliardo-Nirenberg inequality gives that
\be \la{second-level-14}\ba  &\int (B_t \cdot \nabla ) u\cdot B_t dx\\
\leq & \|B_t\|_{L^4}^2 \|\nabla u\|_{L^2}\\
\leq & \frac18 \| B_t\|_{H^1}^2 + C\|\nabla u\|_{L^2}^2 \|B_t\|_{L^2}^2.
\ea \ee
And H\"older's inequality gives that
\be\la{second-level-15}\ba & \int (B\cdot \nabla ) u_t \cdot B_t dx\\
\leq & \frac18\|\nabla u_t\|_{L^2}^2 + C \|B\|_{L^\infty}^2 \|B_t\|_{L^2}^2.
\ea\ee

Collecting all the estimates \eqref{second-level-2}-\eqref{second-level-15} and taking Proposition \ref{First-Level}, \ref{first-level-add-1}, \ref{first-level-add-2} into account,  we get that
\be \la{second-level-16} \ba & \frac12 \frac{d}{dt}  \int |\sqrt{\rho}u_t |^2 dx + \frac12 \frac{d}{dt} \int |B_t|^2 dx + \frac14 \int |\nabla u_t|^2 dx +  \frac14\int |\nabla B_t|^2 dx
\\  \leq &  C(1+ \|u\|_{L^\infty}^4 + \|B\|_{L^\infty}^2 + \|u\|_{H^2}^2 + \|B\|_{H^2}^2) (\|\sqrt{\rho} u_t\|_{L^2}^2 + \|B_t\|_{L^2}^2 )\\
& + C \|\nabla u\|_{L^2}^2 \|u\|_{H^2}^2 + C \|u\|_{L^\infty}^4 \|\nabla u\|_{L^2}^2 ,
\ea\ee
which together with Gronwall's inequality completes the proof of Proposition \ref{Second-Level}.

\end{proof}

As a corollary, we can bound $\|u\|_{L^2_t W^{2,4}_x}$, which will play an important role in the estimates for $\rho$.
\begin{pro}\la{second-level-corollary-1}  Under the assumptions of Theorem \ref{main-result}, it holds that \be \la{second-level-corollary}
\sup_{0< T<  T^*}  \left\{ \|u\|_{L^2(0, T; W^{2,4}) }  \right\} < \infty.
\ee
\end{pro}
\begin{proof}  It follows from Lemma \ref{Galdi} that
\be \ba  &\ \ \ \ \ \|u\|_{W^{2,4}} \\
&  \leq C\|u\|_{H^1} + C \|\rho u_t\|_{L^4} + C\|(\rho u\cdot \nabla) u\|_{L^4} + C\|(B\cdot \nabla) B\|_{L^4}\\
& \leq C\|u\|_{H^1} + C \|\nabla u_t\|_{L^2} + C\|u\|_{L^\infty} \|\nabla u\|_{L^4} + C\|B\|_{L^\infty}\|\nabla B\|_{L^4}\\
& \leq C\|u\|_{H^1} + C\|\nabla u_t\|_{L^2} +C \|u\|_{L^\infty} \|\nabla u\|_{L^2}^{1/2} \|u\|_{H^2}^{1/2} + C\|B\|_{L^\infty}\|\nabla B\|_{L^2}^{1/2} \|B\|_{H^2}^{1/2},
\ea\ee
which finishes the proof of \eqref{second-level-corollary}.

\end{proof}

Furthermore, we have the following proposition.
\begin{pro} Under the assumptions of Theorem \ref{main-result}, it holds that
\la{second-level-corollary-2} \be \sup_{0< T< T^*} \left\{ \|u\|_{H^2}+ \|B\|_{H^2} \right\} < \infty. \ee
\end{pro}
\begin{proof} If  the inequality \eqref{first-level-15} is reconsidered, then the proof is done.

\end{proof}


{\bf Step V\ \ Estimates for $\|\nabla \rho \|_{L^\infty(0, T; H^1)}$ and $ \|(u, B)\|_{L^2(0, T; H^3)}$. }
\begin{pro}Under the assumptions of Theorem \ref{main-result}, it holds that
 \la{Third-Level}
\be\la{third-level}
\sup_{0<T<T^*}\left \{ \|\rho\|_{L^\infty(0, T; H^2)} + \int_0^T \left (\|u\|_{H^3}^2 + \|B\|_{H^3}^2 \right) dt \right\}< \infty.
\ee
\end{pro}

\begin{proof}
Taking the $x_j$ ($j=1, 2$)-derivative of $\eqref{MHD-new}_1$,
\be\la{third-level-1}
(\rho_{x_j})_t + u\cdot \nabla \rho_{x_j} = - u_{x_j} \cdot \nabla \rho.
\ee
Multiplying the new equation by $\rho_{x_j}$, integrating over $\Omega$, and summing up, then we obtain
\be\la{third-level-2}
\frac{d}{dt} \int |\nabla \rho|^2 dx \leq C\int |\nabla u||\nabla \rho |^2 dx \leq C\|\nabla u\|_{L^\infty } \|\nabla \rho\|_{L^2}^2.
\ee
Similarly, we have the following higher order estimate for $\rho$,
\be\la{third-level-3} \ba
\frac{d}{dt} \int |\nabla^2 \rho|^2 dx  & \leq C \int \left( |\nabla u||\nabla^2 \rho|^2 + |\nabla^2 u||\nabla \rho| |\nabla^2 \rho| \right)dx \\
& \leq C \|\nabla u\|_{L^\infty} \|\nabla^2 \rho\|_{L^2}^2 + \|\nabla^2 u \|_{L^4} \|\nabla \rho\|_{L^4} \|\nabla^2 \rho\|_{L^2}.
\ea \ee
Making use of Sobolev embedding inequality and Gronwall's inequality, we get that
\be\la{third-level-4}
\|\nabla \rho(T) \|_{H^1}^2 \leq C \|\nabla \rho_0\|_{H^1}^2 \exp \left( \int_0^T C \|\nabla u(t)\|_{W^{1,4}} dt \right) < \infty.
\ee

It follows from Lemma \ref{Galdi} that
\be\la{third-level-4}\ba
\|u\|_{H^3} & \leq C\left( \|u\|_{H^1} + \|\rho u_t\|_{H^1} + \|\rho u\cdot \nabla u\|_{H^1} + \|B \cdot \nabla B\|_{H^1} \right) \\
& \leq C \left( \|u\|_{H^1} + \|\nabla \rho\|_{L^2} \|u_t\|_{L^2}+ \| u_t\|_{H^1}+\|\nabla \rho\|_{L^2} \|u\|_{L^\infty} \|\nabla u\|_{L^2}    \right) \\
&\ \ \ \ \ \ + C\left(  \|\nabla u\|_{L^2}^2 + \|u\|_{L^\infty} \|\nabla u\|_{H^1}  + \|B\|_{H^1}^2  +\|B\|_{L^\infty} \|\nabla B\|_{H^1} \right)
\ea
\ee
which implies that $\sup_{0< T< T^*} \|u\|_{L^2(0, T; H^3)} < \infty.$ Similar proof leads to the same conclusion for $B$. This completes the proof of Proposition \ref{Third-Level}.
\end{proof}

Combining all the estimates in Proposition \ref{First-Level}, \ref{Second-Level} and \ref{Third-Level}, we prove that \eqref{final} holds and complete the whole proof of Theorem \ref{main-result}.

\end{document}